\documentclass[12pt]{amsart}
\newtheorem{theorem}{Theorem}[section]
\newtheorem{proposition}[theorem]{Proposition}
\newtheorem{lemma}[theorem]{Lemma}
\newtheorem{corollary}[theorem]{Corollary}

\numberwithin{equation}{section}

\begin{document}

\title{Chen--Ruan cohomology of some moduli spaces}

\author[I. Biswas]{Indranil Biswas}

\address{School of Mathematics, Tata Institute of Fundamental
Research, Homi Bhabha Road, Mumbai 400005, India}

\email{indranil@math.tifr.res.in}

\author[M. Poddar]{Mainak Poddar}

\address{Statistics and Mathematics Unit, Indian
Statistical Institute, 203 B. T. Road, Kolkata 700108, India}

\email{mainak@isical.ac.in}



\date{}

\begin{abstract}
Let $X$ be a compact connected Riemann surface of genus
at least two. We compute the Chen--Ruan cohomology ring of
the moduli space of stable $\text{PSL}(2,{\mathbb C})$--bundles
of nontrivial second Stiefel--Whitney class over $X$.
\end{abstract}

\maketitle

\section{Introduction}\label{int}

The Chen--Ruan cohomology ring of an orbifold,
introduced in \cite{CR}, is the degree zero part
of the small quantum cohomology ring of the orbifold
constructed by the same authors \cite{CR2} from
the moduli space of orbifold morphisms of orbifold
spheres into the orbifold. It contains the usual cohomology
ring of the orbifold as a subring. The cohomology
groups associated to it were known earlier in the literature
as orbifold cohomology
groups, primarily due to the work of string
theorists, for orbifolds
that are quotients of a manifold by action
of a finite group. For a large class of compact orbifolds,
namely quotient of a smooth manifold by foliated
action of a compact Lie group,
the Chen--Ruan cohomology group
(with $\mathbb{Z}/2\mathbb{Z}$ grading) is isomorphic
to equivariant
$K$--theory via an equivariant Chern character
map (see \cite{AR}).
If the orbifold has an algebraic structure,
then the Betti numbers of Chen--Ruan cohomology
are invariant under crepant resolutions (see \cite{LP}
and \cite{Ya}),
which underscores their importance in Calabi--Yau geometry.
The ring structure behaves more subtly under resolution, but
is conjectured by Ruan (see \cite{Ru}) to be isomorphic to the
cohomology ring of a smooth crepant resolution if
 both the orbifold and the resolution are
hyper--K\"ahler. This has been proved in
the local case by Ginzburg--Kaledin \cite{GK}, and for
the symmetric product of a projective $K3$ surface by
Fantechi--G\"ottsche \cite{FG}, and Uribe \cite{Ur}. Our aim here
is to compute the Chen--Ruan cohomology ring of a certain
type of moduli spaces of vector bundles which we describe next. 

Let $X$ be a compact connected Riemann surface of genus
$g$, with $g\, \geq\, 2$. Fix a holomorphic line bundle $\xi$
over $X$ such that
$$
\text{degree}(\xi) \, =\, 1\, .
$$

Let ${\mathcal M}_\xi$ denote the moduli space that
parametrizes the isomorphism classes of stable vector
bundles $E$ over $X$ with $\text{rank}(E)\,=\, 2$ and
$\det E\, :=\, \bigwedge^2 E \, =\, \xi$.
This moduli space ${\mathcal M}_\xi$ is an irreducible
complex projective manifold of complex dimension $3g-3$.

Let
\begin{equation}\label{e1}
\Gamma\, :=\, \text{Pic}^0(X)_2\, \subset\, \text{Pic}^0(X)
\end{equation}
be the group of line bundles $L$ over $X$ satisfying the
condition that $L\bigotimes L$ is holomorphically trivial. So
$\Gamma$ is isomorphic to $({\mathbb Z}/2{\mathbb Z})^{\oplus
2g}$. The group $\Gamma$ acts on ${\mathcal M}_\xi$ as follows.

Take any $L\, \in\, \Gamma$. Let
\begin{equation}\label{e2}
\phi_L\, :\, {\mathcal M}_\xi\, \longrightarrow\,
{\mathcal M}_\xi
\end{equation}
be the holomorphic automorphism
defined by $E\, \longmapsto\, E\bigotimes L$. Let
\begin{equation}\label{e3}
\phi\, :\, \Gamma \, \longrightarrow\,
\text{Aut}({\mathcal M}_\xi)
\end{equation}
be the homomorphism defined by $L\, \longmapsto\, \phi_L$.

The quotient space ${\mathcal M}_\xi/ \Gamma$ is the moduli
space of stable $\text{PSL}(2,{\mathbb C})$--bundles $E$ over
$X$ such that the second Stiefel--Whitney class
$$
w_2(E)\, \in\, H^2(X,\, {\mathbb Z}/2{\mathbb Z})
\, =\, {\mathbb Z}/2{\mathbb Z}
$$
is nonzero. For any $E\,\in\, {\mathcal M}_\xi$, the
corresponding $\text{PSL}(2,{\mathbb C})$--bundle is
the one defined by the projective bundle ${\mathbb P}(E)$
associated to $E$.

We compute the Chen--Ruan cohomology ring
of the orbifold ${\mathcal M}_\xi/ \Gamma$.

For each element $L\,\in\, \Gamma$, let
\begin{equation}\label{e12}
{\mathcal S}(L)\, \subset\, {\mathcal M}_\xi
\end{equation}
be the smooth subvariety that is fixed pointwise by
the automorphism $\phi_L$ constructed in \eqref{e2}. Since
$\Gamma$ is abelian, the action of $\Gamma$ on ${\mathcal
M}_\xi$ preserves ${\mathcal S}(L)$.
The Chen--Ruan cohomology group $H_{CR}^* ({\mathcal M}_\xi/
\Gamma ,\, \mathbb Q)$ is defined to be
\begin{equation}\label{eCR}
H_{CR}^* ({\mathcal M}_\xi/ \Gamma , \,\mathbb Q) \,=\,
H^*({\mathcal M}_\xi/ \Gamma, \mathbb Q ) \bigoplus
\left(\bigoplus_{L\in\Gamma\setminus\{{\mathcal O}_X\}}
H^{*-2\iota(L)}
(\mathcal{S}(L)/\Gamma, \, \mathbb Q)\right)\, .
\end{equation}

Note that the first summand is in fact the contribution of
 $\mathcal{S}(L)/\Gamma $ corresponding to $L\, =\,
{\mathcal O}_X$. The degree shift $2\iota(L)$
is a locally constant 
function of the action of $L$ on $T_E{\mathcal M}_\xi$, where
$E \,\in\, \mathcal{S}(L)$. This $\iota(L)$ is
determined by the eigenvalues along with their
multiplicities, of the differential $d\phi_L$.
The ring structure of Chen--Ruan cohomology is
defined via a three point correlation
function (see \eqref{eTP}) and involves
virtual classes or obstruction bundles as in
Gromov--Witten theory. In our situation,
it suffices to know the ranks of these obstruction bundles.

\section{A characterization of the fixed point sets}

Take any holomorphic line bundle $L$ over $X$ such that
$L\bigotimes L$ is holomorphically trivial. Fix a
nonzero holomorphic section
$$
s\, \in\, H^0(X,\, L^{\otimes 2})\, .
$$
Since $L^{\bigotimes 2}$ is trivial, the section $s$ does not
vanish at any point of $X$. Let
\begin{equation}\label{e40}
Y_L\, :=\, \{z\in L\,\mid\, z^{\otimes 2}\, \in\, \text{image}(s)\}
\end{equation}
be the complex projective curve in the total space of $L$. Let
\begin{equation}\label{e4}
\gamma_L\, :\, Y_L\, \longrightarrow\, X
\end{equation}
be the restriction
of the natural projection $L\, \longrightarrow\,X$.
Consider the action of the multiplicative group
${\mathbb C}^*$ on the total space
of $L$. The action of the subgroup
\begin{equation}\label{e4a}
C_2\, :=\, \{z\, \in\, {\mathbb C}\, \mid\, z^2\, =\,1\}
\, \subset\, {\mathbb C}^*
\end{equation}
preserves the curve $Y_L$ in \eqref{e40}. Consequently, $Y_L$
is a principal $C_2$--bundle over $X$. In other words,
the projection $\gamma_L$ makes $Y_L$ an unramified Galois
covering of $X$ with Galois group
\begin{equation}\label{ga.}
\text{Gal}(\gamma_L) \, =\, C_2\, .
\end{equation}
Since any two nonzero sections
of $L$ differ by multiplication with a nonzero constant
scalar, the isomorphism class of the covering $\gamma_L$
does not depend on the choice of the section
$s$. (See \cite[p. 173, Example 3.4]{BNR}.)

Let
\begin{equation}\label{sigma}
\sigma\, :\, Y_L\, \longrightarrow\, Y_L\, \subset\, L
\end{equation}
be the automorphism defined by multiplication with $-1$.

If the line bundle $L$ is nontrivial, then $Y_L$ is connected.
In that case the genus of $Y_L$ is $2g-1$. If $L$ is the
trivial line bundle, then $Y_L$ is the disjoint union of two
copies of $X$, and $\sigma$ in \eqref{sigma} simply
interchanges the two components.

\begin{lemma}\label{lem1}
Let $L$ be a nontrivial holomorphic line bundle over $X$ of
order two. Take a holomorphic
line bundle $\eta$ over $Y_L$ (see \eqref{e4})
of degree one. Then the direct
image
$$
\gamma_{L*} \eta\, \longrightarrow\, X
$$
is a stable
vector bundle over $X$ of rank two and degree one.
\end{lemma}

\begin{proof}
Since the covering $\gamma_L$ is unramified,
$$
\text{degree}(\gamma_{L*} \eta)\, =\, \text{degree}(\eta)
\, =\, 1\, .
$$
We note that
\begin{equation}\label{e-1}
\gamma^*_L \gamma_{L*} \eta\, =\, \eta \bigoplus \sigma^*\eta\, ,
\end{equation}
where $\sigma$ is defined in \eqref{sigma}. Since
$\text{degree}(\sigma^*\eta)\, =\, \text{degree}(\eta)$,
the right--hand side in \eqref{e-1} is a polystable vector
bundle on $Y_L$. Consequently, the vector
bundle $\gamma_{L*} \eta$ is polystable.
Now we conclude that $\gamma_{L*} \eta$ is stable because
$\text{rank}(\gamma_{L*} \eta)$
is coprime to $\text{degree}(\gamma_{L*} \eta)$.
\end{proof}

Fix a holomorphic line bundle $\xi$ over $X$ of degree one.
As before, by ${\mathcal M}_\xi$ we denote the moduli space of
stable vector bundles of rank two over $X$ with
$\bigwedge^2E\, =\,\xi$.

\begin{proposition}\label{prop1}
Let $L\, \in\, \Gamma \setminus\{{\mathcal O}_X\}$ be a
line bundle over $X$ of order two.
\begin{enumerate}
\item Take any $\eta\, \in\, {\rm Pic}^1(Y_L)$, where $Y_L$
is constructed in \eqref{e4}, such that the line bundle
$\bigwedge^2\gamma_{L*} \eta$ is isomorphic to $\xi$. Then
$\gamma_{L*} \eta\, \longrightarrow\, X$ is a fixed point of
the automorphism $\phi_L$ constructed in \eqref{e2}.

\item Let $E\, \in\, {\mathcal M}_\xi$ be such that
$\phi_L(E) \, =\, E$, where $\phi_L$ is the map
in \eqref{e2}. Then there is a
holomorphic line bundle $\eta$ over $Y_L$ (see \eqref{e4}) such
that the direct image $\gamma_{L*} \eta$ is isomorphic to $E$.

\item Let $\eta_1$ and $\eta_2$ be holomorphic line bundles over
$Y_L$ of degree one.
Then the direct image $\gamma_{L*} \eta_1$ is isomorphic to
$\gamma_{L*} \eta_2$ if and only if there is a unique element
$$
\tau\, \in\, {\rm Gal}(\gamma_L)\, =\, {\mathbb Z}/2\mathbb Z
$$
of the Galois group for $\gamma_L$ such that $\eta_1\,=\,
\tau^*\eta_2$.
\end{enumerate}
\end{proposition}

\begin{proof}
{}From Lemma \ref{lem1} we know that for any $\eta'\,
\in\, {\rm Pic}^1(Y_L)$, the direct image $\gamma_{L*}
\eta'$ is stable. Hence $\gamma_{L*} \eta$ in the
first part of the proposition lies in ${\mathcal M}_\xi$.
The pull back of any line bundle $L_1$ to the complement
of the zero section of $L_1$ has a canonical
trivialization. In particular, the pull back
$\gamma^*_L L$ has a canonical trivialization.
Therefore, we have a natural isomorphism
$$
h\, :\, \eta\, =\, \eta\bigotimes\nolimits_{{\mathcal
O}_X}{\mathcal O}_X
\, \longrightarrow\, \eta\bigotimes \gamma^*_L L
$$
which is
obtained by tensoring $\text{Id}_\eta$ with the
homomorphism ${\mathcal O}_X\, \longrightarrow\,
\gamma^*_L L$ defining the trivialization of $\gamma^*_L L$.
The above isomorphism $h$ induces an isomorphism
\begin{equation}\label{e5a}
\gamma_{L*}h\, :\, \gamma_{L*}\eta\, \longrightarrow\,
\gamma_{L*}(\eta\bigotimes \gamma^*_L L)\, =\,
(\gamma_{L*}\eta)\bigotimes L
\end{equation}
with the isomorphism $\gamma_{L*}(\eta\bigotimes \gamma^*_L L)
\, =\,(\gamma_{L*}\eta)\bigotimes L$
being given by the projection
formula. Hence $\gamma_{L*}\eta$ is a fixed point of the
automorphism $\phi_L$ in \eqref{e2}.
This proves statement (1) in the proposition.

Take any $E\, \in\, {\mathcal M}_\xi$
such that $\phi_L(E) \, =\, E$.
Fix a holomorphic isomorphism of vector bundles
\begin{equation}\label{e5}
f\, :\, E\, \longrightarrow\, E\bigotimes L\, .
\end{equation}
For each $i\, \in\, \{1\, ,2\}$, we have
$$
\text{trace}(f^i)\, \in\, H^0(X,\, L^{\otimes i})
$$
(see \cite[\S~3]{BNR}, \cite{Hi}). Also, note that
$H^0(X,\, L)\, =\, 0$ because $L$ is nontrivial.
Therefore, the spectral curve
for the pair $(E\, ,f)$ is the covering $Y_L$ in \eqref{e4}.

There is a holomorphic line bundle $\eta$ over $Y_L$ such that
$\gamma_{L*} \eta$ is isomorphic to $E$ \cite[\S~3]{BNR},
\cite{Hi}, where $\gamma_L$ is the map in \eqref{e4}.
This proves statement (2) in the proposition.

To prove statement (3), take $\eta_1$
and $\eta_2$ as in that statement of the proposition.
If $\eta_1\,=\,\tau^*\eta_2$ for some
$\tau\, \in\, {\rm Gal}(\gamma_L)$, then clearly
$\gamma_{L*} \eta_1$ is isomorphic to $\gamma_{L*} \eta_2$.

Now assume that $\gamma_{L*} \eta_1$ is isomorphic to $\gamma_{L*}
\eta_2$. Fix an isomorphism
\begin{equation}\label{e6}
\alpha\, :\, E_1\, :=\, \gamma_{L*} \eta_1\, \longrightarrow\,
\gamma_{L*} \eta_2\, :=\, E_2\, .
\end{equation}
We now note that
\begin{equation}\label{e7}
\gamma_{L*}\left(\bigoplus_{\tau\in
\text{Gal}(\gamma_L)}\eta^*_1\bigotimes \tau^*\eta_2\right)
\,=\, \bigoplus_{\tau\in\text{Gal}(\gamma_L)}
\gamma_{L*}(\eta^*_1\bigotimes \tau^*\eta_2)
\, =\, E^*_1\bigotimes E_2\, .
\end{equation}

Since $\gamma_L$ is a finite morphism, for any holomorphic
vector bundle $W$ on $Y_L$,
\begin{equation}\label{isi1}
H^i(Y_L,\, W) \, =\, H^i(X,\, \gamma_{L*}W)
\end{equation}
for all $i$. Therefore, from \eqref{e7},
\begin{equation}\label{e8}
H^0(X,\, {\mathcal H}om(E_1\, ,E_2))\, =\, H^0(Y_L,\,
{\mathcal H}om (\eta_1\, ,\eta_2))\bigoplus H^0(Y_L,\,
{\mathcal H}om (\eta_1\, ,\sigma^*\eta_2))\, ,
\end{equation}
where $\sigma$ is the automorphism in \eqref{sigma}.
Consequently, the nonzero element
$$
\alpha\, \in\, H^0(X,\, {\mathcal H}om(E_1\, ,E_2))
$$
in \eqref{e6} gives
a nonzero element in the right--hand side of \eqref{e8}.
Hence we conclude that either $\eta_1$ is isomorphic to
$\eta_2$ or $\eta_1$ is isomorphic to $\sigma^*\eta_2$.

To complete the proof of statement (3) we need to show that
$\eta_1$ can not be isomorphic to both $\eta_2$ and
$\sigma^*\eta_2$.

If $\eta_1\, =\, \eta_2\, =\, \sigma^*\eta_2$, then from
\eqref{e8} we conclude that
\begin{equation}\label{nz2}
\dim H^0(X,\, {\mathcal H}om(E_1\, ,E_2))\, \geq\, 2\, .
\end{equation}
On the other hand, both $E_1$ and $E_2$ are stable vector
bundles over $X$ of rank $r$ and degree one (see Lemma
\ref{lem1}). Hence
$$
\dim H^0(X,\, {\mathcal H}om(E_1\, ,E_2))\, \leq\, 1\, .
$$
But this contradicts \eqref{nz2}. Therefore,
$\eta_2\, \not=\, \sigma^*\eta_2$.
This completes the proof of the proposition.
\end{proof}

\section{Tangential action at fixed points}\label{eigen}

The holomorphic tangent bundle of ${\mathcal M}_\xi$ will
be denoted by $T{\mathcal M}_\xi$.

Let $L$ be any nontrivial line bundle over $X$ of order two.
Take any stable vector bundle $E\, \in\, {\mathcal M}_\xi$
such that $\phi_L(E)\, =\, E$, where $\phi_L$ is
constructed in \eqref{e2}. The following lemma describes
the spectral decomposition of the differential
\begin{equation}\label{d}
d\phi_L(E)\, :\, T_E {\mathcal M}_\xi\, \longrightarrow\,
T_E{\mathcal M}_\xi
\end{equation}
at the point $E\, \in\, {\mathcal M}_\xi$; here $T_E
{\mathcal M}_\xi$ is the fiber of $T{\mathcal M}_\xi$ at $E$.

\begin{lemma}\label{lem2}
The eigenvalues of the differential $d\phi_L(E)$ in \eqref{d}
are $\pm 1$. The multiplicity of the eigenvalue $1$ is $g-1$.
The multiplicity of the eigenvalue $-1$
is $2(g-1)$.
\end{lemma}

\begin{proof}
Since $\phi_L\circ \phi_L\, =\, \text{Id}_{{\mathcal M}_\xi}$,
the only possible eigenvalues of $d\phi_L(E)$ are $-1$ and
$1$.

Proposition \ref{prop1}(2) says that
there is a holomorphic line bundle $\eta$ on $Y_L$ such that
$$
E_\eta\, :=\, \gamma_{L*}\eta \, \cong\, E\, .
$$

Consider the isomorphism $\gamma_{L*}h$ constructed in
\eqref{e5a}. For any vector bundle $W$ over $X$, the endomorphism
bundle ${\mathcal E}nd(W\bigotimes L)\, =\, (W\bigotimes L)
\bigotimes (W\bigotimes L)^*$ is canonically identified with
${\mathcal E}nd(W)\, =\, W\bigotimes W^*$. Hence the isomorphism
$\gamma_{L*}h$ of $\gamma_{L*}\eta$ with $(\gamma_{L*}\eta)
\bigotimes L$ defines an automorphism of the vector bundle
${\mathcal E}nd(\gamma_{L*}\eta)$
\begin{equation}\label{aut.}
\theta\, :\, {\mathcal E}nd(\gamma_{L*}\eta)\, \longrightarrow
\,{\mathcal E}nd(\gamma_{L*}\eta)\, .
\end{equation}
Let
\begin{equation}\label{ad}
\text{ad}(E_\eta) \, =\, \text{ad}(\gamma_{L*}\eta)\, \subset\,
{\mathcal E}nd (\gamma_{L*}\eta)
\end{equation}
be the subbundle of corank one given by the sheaf of
endomorphisms of $E_\eta$ of trace zero. It is easy to see that
$\theta$ in \eqref{aut.} preserves this subbundle $\text{ad}
(\gamma_{L*}\eta)$. Hence $\theta$ induces an automorphism
\begin{equation}\label{e5b}
\theta_0\, :\, \text{ad}(\gamma_{L*}\eta)\, \longrightarrow\,
\text{ad}(\gamma_{L*}\eta)
\end{equation}
of the vector bundle $\text{ad}(\gamma_{L*}\eta)$. Let
\begin{equation}\label{e5c}
\overline{\theta}_0\,:\,H^1(X,\, \text{ad}
(\gamma_{L*}\eta))\,\longrightarrow\,
H^1(X,\, \text{ad}(\gamma_{L*}\eta))
\end{equation}
be the automorphism induced by $\theta_0$ in \eqref{e5b}.

The tangent space $T_E {\mathcal M}_\xi$ is identified
with $H^1(X,\, \text{ad}(\gamma_{L*}\eta))$. The differential
$d\phi_L(E)$ in \eqref{d} coincides with the automorphism
$\overline{\theta}_0$ constructed in \eqref{e5c}.

{}From \eqref{e7} we know that
\begin{equation}\label{7a}
\gamma_{L*}\left((\eta^*\bigotimes\eta) \bigoplus
(\eta^*\bigotimes \sigma^*\eta)\right)\, =\,
E^*_\eta\bigotimes E_\eta\, =\, {\mathcal E}nd(E_\eta)\, ,
\end{equation}
where $E_\eta\, =\,\gamma_{L*}\eta$, and
$\sigma$ is defined in \eqref{sigma}. From \eqref{7a}
and \eqref{isi1},
\begin{equation}\label{7b}
H^1(X,\, {\mathcal E}nd(E_\eta))\, =\,
H^1(Y_L,\, {\mathcal H}om
(\eta\, ,\eta))\bigoplus
H^1(Y_L,\, {\mathcal H}om
(\eta\, ,\sigma^*\eta))
\end{equation}
(as in \eqref{e8}).

Consider the nontrivial
element $\sigma\, \in\, \text{Gal}(\gamma_L)\, =\,
C_2$ (see \eqref{sigma}). The automorphism
$\theta$ of ${\mathcal E}nd(\gamma_{L*}\eta)$ in \eqref{aut.}
preserves the subbundle
$$
\gamma_{L*}(\eta^*\bigotimes \sigma^*\eta)\, \subset\,
{\mathcal E}nd(E_\eta)
$$
in \eqref{7a}, and furthermore, $\theta$ acts on this
subbundle $\gamma_{L*}(\eta^*\bigotimes \sigma^*\eta)$ as
multiplication by $-1$. It is easy see that
\begin{equation}\label{e13a}
\gamma_{L*}(\eta^*\bigotimes \sigma^*\eta)\, \subset\,
\text{ad}(E_\eta)\, \subset\, {\mathcal E}nd(E_\eta)\, .
\end{equation}
We also note that the automorphism
$$
\theta\, \in\, \text{Aut}(
{\mathcal E}nd(E_\eta))
$$
acts trivially on the
subspace
$$
\gamma_{L*}(\eta^*\bigotimes\eta)\, \subset\,
{\mathcal E}nd(E_\eta)
$$
in \eqref{7a}. Therefore, the subspace of
$H^1(X,\, \text{ad}(\gamma_{L*}\eta))$ on which
the automorphism $\overline{\theta}_0$ in \eqref{e5c}
acts as multiplication by $-1$ coincides with
the subspace
$$
H^1(Y_L,\, {\mathcal H}om
(\eta\, ,\sigma^*\eta))\, \subset\,
H^0(X,\, \text{ad}(E_\eta))
$$
in \eqref{7b}.

{}From \eqref{isi1} we have
$H^i(X,\, \gamma_{L*}(\eta^*\bigotimes \sigma^*\eta))\, =\,
H^i(Y_L,\, \eta^*\bigotimes \sigma^*\eta)$
for all $i$. From \eqref{e13a},
$$
H^0(Y_L,\, \eta^*\bigotimes \sigma^*\eta)\, \subset\,
H^0(X,\, \text{ad}(E_\eta))\, .
$$
But $H^0(X,\, \text{ad}(E_\eta))\, =\, 0$ because
the vector bundle $E_\eta$ is stable (see Lemma \ref{lem1}).
Hence
\begin{equation}\label{e14}
H^0(Y_L,\, \eta^*\bigotimes \sigma^*\eta)\, =\, 0\, .
\end{equation}
Since $\text{genus}(Y_L)\, =\, 2g-1$,
using Riemann--Roch, from \eqref{e14} it follows that
$$
\dim H^1(X,\, \gamma_{L*}(\eta^*\bigotimes \sigma^*\eta))\,=\,
2(g-1)\, .
$$

Therefore, $-1$ is
an eigenvalue of the automorphism $\theta_0$ in \eqref{e5b}
of multiplicity $2(g-1)$.

We already noted that the only possible eigenvalues of
$d\phi_L(E)$ are $-1$ and $1$. Hence $1$ is
an eigenvalue of the automorphism $\theta_0$ in \eqref{e5b}
of multiplicity $g-1$. This completes the proof of the lemma.
\end{proof}

\begin{corollary}\label{cor1}
The degree shift $\iota(L) \,=\, g-1$ when $L\in \Gamma$ is
nontrivial, and $\iota(L)\, =\, 0$ when $L$ is trivial.
\end{corollary}

\begin{proof}
If the eigenvalues are $\exp(2\pi \sqrt{-1} a_j)$, where $0
\,\le\, a_j \,< \,1$ with multiplicity $m_j$, then by definition
$$
\iota(L)= \sum_j a_j m_j \, .
$$
So the corollary follows immediately from Lemma \ref{lem2}.
\end{proof}

\section{Intersection of fixed point sets}

Take any $L\, \in\, \Gamma\setminus \{{\mathcal O}_X\}$
(see \eqref{e1}).
Consider the covering $\gamma_L$ in \eqref{e4}
associated to $L$. Since the Galois group for $\gamma_L$
is ${\mathbb Z}/2{\mathbb Z}$, the covering
$\gamma_L$ defines a surjective homomorphism
$$
H_1(X,\, {\mathbb Z}) \, \longrightarrow\,
{\mathbb Z}/2{\mathbb Z}\, .
$$
Such a homomorphism gives a nonzero element in
$H^1(X,\, {\mathbb Z}/2{\mathbb Z})$.

Let
\begin{equation}\label{omega}
\omega\, :\, \Gamma\, :=\, \text{Pic}^0(X)_2\,
\longrightarrow\, H^1(X,\, {\mathbb Z}/2{\mathbb Z})
\end{equation}
be the homomorphism that sends any $L$ to the cohomology
class constructed above from it. This homomorphism $\omega$
is in fact an isomorphism. Let
\begin{equation}\label{mu}
\mu\, :\, H^1(X,\, {\mathbb Z}/2{\mathbb Z})
\bigotimes\nolimits_{{\mathbb Z}/2{\mathbb Z}}
H^1(X,\, {\mathbb Z}/2{\mathbb Z})\, \longrightarrow\,
H^2(X,\, {\mathbb Z}/2{\mathbb Z})\, =\, {\mathbb Z}/2{\mathbb Z}
\end{equation}
be the cup product. It is known that the isomorphism $\omega$
in \eqref{omega} takes $\mu$ to the Weil--pairing on
$\text{Pic}^0(X)_2$ (see \cite[p. 183]{Mu1} for the
definition of Weil--pairing).

Fix two nontrivial holomorphic line bundles $L$ and $L'$
over $X$ of order two such that $L\, \not=\, L'$.
Let ${\mathcal S}(L)$ and ${\mathcal S}(L')$ be the
corresponding subvarieties of ${\mathcal M}_\xi$
parametrizing the fixed point sets of $\phi_L$ and
$\phi_{L'}$ respectively (see \eqref{e12}).

\begin{proposition}\label{prop int}
Let $L$ and $L'$ be nontrivial
line bundles of order two over $X$ such that $L$ is
not isomorphic to $L'$.
The variety ${\mathcal S}(L)$ does not intersect with
${\mathcal S}(L')$ if
$$
\mu(\omega(L)\bigotimes \omega(L'))\, =\, 0\, ,
$$
where $\omega$ and $\mu$ are defined in \eqref{omega}
and \eqref{mu} respectively.

If
$$
\mu(\omega(L)\bigotimes \omega(L'))\, \not=\, 0
$$
then ${\mathcal S}(L)\bigcap {\mathcal S}(L')$ is a finite
set of cardinality $2^{2g-2}$.
\end{proposition}

\begin{proof}
Take any vector bundle $E\, \in\, {\mathcal S}(L)$. Let
\begin{equation}\label{ad2}
\text{ad}(E)\, \subset\, {\mathcal E}nd (E)
\end{equation}
be the subbundle
of corank one defined by the sheaf of trace zero endomorphisms
of $E$. Since $E\, \in\, {\mathcal S}(L)$, the vector bundle
$E\bigotimes L$ is holomorphically isomorphic to $E$. Fix a
holomorphic isomorphism
$$
A \, :\, E\bigotimes L\, \longrightarrow\, E\, .
$$
This isomorphism $A$ defines a holomorphic homomorphism
\begin{equation}\label{vp}
\varpi\, :\, L\, \longrightarrow\, {\mathcal E}nd (E)
\end{equation}
of coherent sheaves. Now consider the composition
$$
L\, \stackrel{\varpi}{\longrightarrow}\, {\mathcal E}nd (E)
\, \stackrel{\rm trace}{\longrightarrow}\, {\mathcal O}_X\, .
$$
Since $L$ is a nontrivial line bundle of degree zero, there
is no nonzero holomorphic
homomorphism from $L$ to ${\mathcal O}_X$.
Hence the above composition of homomorphisms vanishes
identically.
Therefore, we conclude that the homomorphism $\varpi$ in
\eqref{vp} makes $L$ a coherent subsheaf of
$\text{ad}(E)$ defined in \eqref{ad2}.

Take any $E\, \in\, {\mathcal S}(L)\bigcap {\mathcal S}(L')$.
Given isomorphisms $E\, \stackrel{\alpha}{\longrightarrow}\,
E\bigotimes L$ and $E\, \stackrel{\beta}{\longrightarrow}\,
E\bigotimes L$, we have the composition isomorphism
$$
E\,\stackrel{\beta}{\longrightarrow}\,E\bigotimes
L' \, \stackrel{\alpha\otimes\text{Id}_{L'}}{
\longrightarrow}\, E\bigotimes L\bigotimes L'\, .
$$
Consequently,
$$
E\, \in\, {\mathcal S}(L\bigotimes L')\, .
$$
Therefore, we have
an injective homomorphism of coherent sheaves
\begin{equation}\label{cell}
{\mathcal E}(L,L')\, :=\,
L\bigoplus L' \bigoplus (L\bigotimes L')\,
\longrightarrow\, \text{ad}(E)\, .
\end{equation}
Since $\text{degree}({\mathcal E}(L,L'))
\, =\, \text{degree}(\text{ad}(E))$ (both are zero),
this injective homomorphism
must be an isomorphism. Therefore, we conclude that
\begin{equation}\label{vp2}
{\mathcal E}(L,L')\,=\, \text{ad}(E)\, ,
\end{equation}
where ${\mathcal E}(L,L')$ is defined in \eqref{cell}.

Fix trivializations of $L\bigotimes L$ and
$L'\bigotimes L'$. These two trivializations
together give a trivialization
of $(L\bigotimes L')^{\otimes 2}$. The three
trivializations together give a Lie algebra structure
on the fibers of the vector bundle ${\mathcal E}(L,L')$
(see \eqref{cell}) defined by
\begin{equation}\label{mul.}
[(a\, ,b\, ,c)\, ,(a'\, ,b'\, ,c')]\, :=\, 2\cdot
((b'\otimes c)-(b\otimes c')\, , (a'\otimes c)
-(a\otimes c')\, ,(a'\otimes b)-(a\otimes b'))\, .
\end{equation}
Given any holomorphic automorphism $T$ of the vector
bundle ${\mathcal E}(L,L')$ over $X$, we get
a new Lie algebra structure on the fibers
of ${\mathcal E}(L,L')$ by transporting the earlier
Lie algebra structure using $T$. These Lie
algebra structures together
define an equivalence class of Lie algebra structures
on the fibers of ${\mathcal E}(L,L')$.
It is straight forward
to check that for each point $x\, \in\, X$, the
Lie algebra ${\mathcal E}(L,L')_x$
defined above is isomorphic to $\text{sl}(2,{\mathbb C})$;
to see this use the basis
$$
\begin{pmatrix}
0 & 1\\
1 & 0
\end{pmatrix}\, ,
\begin{pmatrix}
0 & -1\\
1 & 0
\end{pmatrix}
\,~~\text{and}\, ~~
\begin{pmatrix}
-1 & 0\\
0 & 1
\end{pmatrix}
$$
of $\text{sl}(2,{\mathbb C})$.

Consider the Lie algebra structure of the fibers
of $\text{ad}(E)$ constructed using the composition of
endomorphisms of $E$.
Since $L$, $L'$ and $L\bigotimes L'$ are all distinct line
bundles, and all are different from the trivial line bundle,
it can be shown that the isomorphism in \eqref{vp2} takes
this Lie algebra structure of the fibers of $\text{ad}(E)$
to the above mentioned equivalence class given by the Lie
algebra structure constructed in \eqref{mul.}.

We also note that the
group of all holomorphic automorphisms of
the vector bundle ${\mathcal E}(L,L')\, :=\,
L\bigoplus L' \bigoplus (L\bigotimes L')$
coincides with ${\mathbb C}^*\times {\mathbb C}^*\times{\mathbb
C}^*$ with ${\mathbb C}^*$ acting as automorphisms of each
direct summand.

Consider the projective bundle
\begin{equation}\label{cP}
{\mathcal P}\, \longrightarrow\, X
\end{equation}
of relative
dimension one defined by the projectivized nonzero
nilpotent elements
in the fibers of ${\mathcal E}(L,L')$. So for
each point $x\,\in\, X$, the fiber ${\mathcal P}_x$ of
${\mathcal P}$ over $x$ is the projectivization of all elements
$$
(a\, ,b\, ,c)\, \in\, {\mathcal E}(L,L')_x
$$
such that
$$
a^2-b^2+c^2\, =\, 0\, .
$$
Note that since $a\, \in\, L_x$, $b\, \in\, (L')_x$ and
$c\, \in\, (L\bigotimes L')_x$,
using the trivializations of $L^{\otimes 2}$,
$(L')^{\otimes 2}$ and $(L\bigotimes L')^{\otimes 2}$,
we have $a^2\, ,b^2\, , c^2\, \in\, {\mathbb C}$.

We noted above that the isomorphism in \eqref{vp2} takes the
natural Lie algebra structure of the fibers of $\text{ad}(E)$
to the equivalence class given by the Lie algebra
structure defined in \eqref{mul.}. Using this it
can be deduced that the projective bundle
${\mathbb P}(E)$ over $X$ is isomorphic
to ${\mathcal P}$ constructed in \eqref{cP}. Indeed, this
follows from the above observation and 
the fact that for any complex vector
space $W_0$ of dimension two, the space of all
projectivized nonzero nilpotent elements
in $\text{End}_{\mathbb C}(W_0)$ is canonically
identified with ${\mathbb P}(W_0)$. The identification
sends a nilpotent endomorphism $N$ to the line in $W_0$
defined by the image of $N$.

The projective bundle $\mathcal P$ defines a
holomorphic principal $\text{PGL}(2,{\mathbb C})$--bundle
over $X$. Let $\text{ad}({\mathcal P})$ be the associated
adjoint vector bundle. We recall that
$\text{ad}({\mathcal P})$ is the vector bundle associated
to the principal $\text{PGL}(2,{\mathbb C})$--bundle
$\mathcal P$ for the adjoint action of
$\text{PGL}(2,{\mathbb C})$ on its own Lie algebra
$\text{sl}(2,{\mathbb C})$. It is easy to see
that $\text{ad}({\mathcal P})$ coincides with the
direct image of the relative tangent bundle on
the total space of the
projective bundle $\mathcal P$. Since
$\mathcal P$ is identified with the projective bundle
${\mathbb P}(E)$, it follows immediately that
\begin{equation}\label{adid}
\text{ad}({\mathcal P})\, =\, \text{ad}({\mathbb P}(E))
\,=\, {\mathcal E}(L,L')\, ,
\end{equation}
where ${\mathcal E}(L,L')$ is the vector bundle
defined in \eqref{cell}.

Consider the second Stiefel--Whitney class
$$
w_2({\mathcal P})\, \in\, H^2(X,\, {\mathbb Z}/2{\mathbb Z})
\, =\, {\mathbb Z}/2{\mathbb Z}
$$
of the projective bundle ${\mathcal P}$
defined in \eqref{cP}. Since ${\mathcal P}\, =\,
{\mathbb P}(E)$, using \eqref{vp2} and \eqref{cell} it
follows that $w_2({\mathcal P})$ coincides with
$$
\mu(\omega(L)\bigotimes \omega(L'))\, \in\,
{\mathbb Z}/2{\mathbb Z}
$$
where $\omega$ and $\mu$ are defined in \eqref{omega}
and \eqref{mu} respectively. Therefore, if $V$ is a complex
vector bundle of rank two over $X$ such that the projective
bundle ${\mathbb P}(V)$ is isomorphic to ${\mathcal P}$, then
\begin{equation}\label{deg.}
\text{degree}(V)\, \equiv\, \mu(\omega(L)\bigotimes\omega(L'))
~\,~\,~\,{\text{~~}} (\text{mod~}\,~\, 2)\, .
\end{equation}

Since ${\mathbb P}(E)$ over $X$ isomorphic
to ${\mathcal P}$, from \eqref{deg.} we have
\begin{equation}\label{deg2.}
1\, =\, \text{degree}(E)\, \equiv\,
\mu(\omega(L)\bigotimes\omega(L'))
~\,~\,~\,{\text{~~}} (\text{mod~}\,~\, 2)\, .
\end{equation}

If $\mu(\omega(L)\bigotimes\omega(L'))\, \in\,
{\mathbb Z}/2\mathbb Z$ vanishes, the two sides of
\eqref{deg2.} are different. Therefore, we conclude that
$$
{\mathcal S}(L)\bigcap {\mathcal S}(L')\, =\, \emptyset
$$
whenever $\mu(\omega(L)\bigotimes\omega(L'))\, =\, 0$.

Now we assume that
\begin{equation}\label{mui1}
\mu(\omega(L)\bigotimes\omega(L'))\, =\, 1
\, \in\, {\mathbb Z}/2{\mathbb Z}\, .
\end{equation}

Define ${\mathcal E}(L,L')$ as in \eqref{cell},
and define the Lie algebra structure as in
\eqref{mul.}. Construct the projective bundle
$\mathcal P$ as in \eqref{cP} from this Lie
algebra bundle. We noted earlier that
$w_2({\mathcal P})$ coincides with
$\mu(\omega(L)\bigotimes\omega(L'))$. Hence from
\eqref{mui1} it follows
that $w_2({\mathcal P})\, \not=\, 0$. Consequently,
there is a holomorphic vector bundle $V$ over $X$ of
rank two and odd degree
such that ${\mathbb P}(V)\, =\, {\mathcal P}$.
Fix a holomorphic line bundle $L_0$ over $X$ such that
$$
L^{\otimes 2}_0\bigotimes \bigwedge\nolimits^2 V\, =\, \xi\, .
$$
Therefore,
\begin{equation}\label{deg3.}
E_0\, :=\, V\bigotimes L_0
\end{equation}
is a holomorphic vector bundle over $X$ of rank two such
that $\bigwedge^2 E_0 \, =\, \xi$ and ${\mathbb P}(E_0)$
is isomorphic to $\mathcal P$.

The isomorphism in \eqref{adid} holds. Therefore, from
the fact that ${\mathbb P}(E_0)$ is isomorphic to $\mathcal P$,
we conclude that $L\bigoplus L'$ is a direct summand of
the vector bundle $\text{ad}(E_0)$. Consequently, we have
$$
E_0\, \in\, {\mathcal S}(L)\bigcap {\mathcal S}(L')\, .
$$
If $E_1\, \in\, {\mathcal S}(L)\bigcap
{\mathcal S}(L')$, then we have
$$
{\mathbb P}(E_1)\, =\, {\mathcal P}\, =\, 
{\mathbb P}(E_0)\, .
$$
Hence a vector bundle $E_1\, \in\, {\mathcal M}_\xi$ lies
in ${\mathcal S}(L)\bigcap{\mathcal S}(L')$ if and only if
$$
E_1\, =\, E_0\bigotimes L_1\, ,
$$
where $L_1\in\, \Gamma$ (see \eqref{e1}), and $E_0$ is
constructed in \eqref{deg3.}.

On the other hand,
$$
E_0\bigotimes L\, =\, E_0\,= \, E_0\bigotimes L'
$$
because $E_0\, \in\, {\mathcal S}(L)\bigcap {\mathcal
S}(L')$. It can be shown that for any nontrivial
holomorphic line bundle $L''\, \in\, \Gamma$
which is different from the three line bundles
$L$, $L'$ and $L\bigotimes L'$, the vector
bundle $E_0\bigotimes L''$ is not isomorphic to
$E_0$. Indeed, if $E_0\bigotimes L''$ is isomorphic
to $E_0$, then from \eqref{vp2} we know that $L''$ is a
direct summand of $\text{ad}(E_0)
\,=\, L\bigoplus L'\bigoplus (L\bigotimes L')$.
Hence from the uniqueness of decomposition of
a vector bundle (see \cite[p. 315, Theorem 3]{At})
it follows immediately that the holomorphic line
bundle $L''$ must be isomorphic to one of $L$, $L'$
and $L\bigotimes L'$. Therefore, we conclude that
$E_0\bigotimes L''$ is not isomorphic to $E_0$ if
$L''$ is different from the three line bundles
$L$, $L'$ and $L\bigotimes L'$.

Consequently, the intersection ${\mathcal S}(L)\bigcap
{\mathcal S}(L')$ is a affine space for the
quotient group of $\Gamma$ obtained by quotienting
it with the subgroup generated by $L$ and $L'$.
In particular, we have
$$
\#({\mathcal S}(L)\bigcap {\mathcal S}(L')) \,=\,
(\#\Gamma)/4 \, =\, 2^{2g}/4\, =\, 2^{2(g-1)}\, .
$$
This completes the proof of the proposition.
\end{proof}

\section{Cohomology groups}\label{sec4}

\subsection{Cohomology of $\text{Gal}(\gamma_L)
\backslash{\mathcal S}(L)/\Gamma$}

Take any $L\, \in\, \Gamma\setminus\{{\mathcal O}_X\}$
(see \eqref{e1}). Since $\Gamma$ is abelian, it acts
on the fixed point set ${\mathcal S}(L)$ defined in
\eqref{e12}. We will compute the cohomology groups
of the quotient space ${\mathcal S}(L)/\Gamma$.

Let $W_0$ be a $\mathbb Q$--vector space of dimension
$2(g-1)$.
Consider the following action of the group
$C_{2}\,=\, \{\pm 1\}$ on
$W_0$: the element $-1\, \in\, C_{2}$ acts as
multiplication by $-1$. This action of $C_{2}$ on $W_0$
induces an action of $C_{2}$
on the exterior algebra $\bigwedge W_0$.

For any even integer $i$, define
\begin{equation}\label{e11}
d_g(i)\, :=\, \dim
(\bigwedge\nolimits^i W_0)^{C_{2}}
\, =\, \binom{2g-2}{i}\, ,
\end{equation}
in particular, $d_g(0)\, =\, 1$, and for any odd positive
integer $i$, define
\begin{equation}\label{e11n}
d_g(i)\, :=\,0\, .
\end{equation}

Let
\begin{equation}\label{e15}
\text{Prym}(\gamma_L)\, \subset\, \text{Pic}^1(Y_L)
\end{equation}
be the Prym variety parametrizing all line bundles $\eta$
over $Y$ such that
$$
\bigwedge\nolimits^2 \gamma_{L*}\eta\, =\, \xi
$$
(see \cite{BNR}, \cite{Hi}, \cite{Mu2}).
It is known that $\text{Prym}(\gamma_L)$ is a
complex abelian variety of dimension $g-1$.

\begin{proposition}\label{prop2}
For any positive integer $i$,
$$
\dim H^i({\mathcal S}(L)/\Gamma,\, {\mathbb Q})\,=\, d_g(i)\, ,
$$
where $d_g(i)$ is defined in \eqref{e11} and \eqref{e11n}.
More precisely, the vector space
$H^i({\mathcal S}(L)/\Gamma,\, {\mathbb Q})$ is identified
with $(\bigwedge\nolimits^i W_0)^{C_{2}}$, where
$W_0\, =\, H^1({\rm Prym}(\gamma_L),\, {\mathbb Q})$.
\end{proposition}

\begin{proof}
Let
\begin{equation}\label{e16}
p_0\, :\, \text{Prym}(\gamma_L)\, \longrightarrow\,
{\mathcal M}_\xi
\end{equation}
be the morphism defined by $\eta\, \longmapsto\,
\gamma_{L*}\eta$ (see Lemma \ref{lem1}). Using
Proposition \ref{prop1}(1) if follows that
$$
p_0(\text{Prym}(\gamma_L))\, \subset\,{\mathcal S}(L)\, ,
$$
where ${\mathcal S}(L)$ is defined in \eqref{e12}. From
Proposition \ref{prop1}(2),
$$
p_0(\text{Prym}(\gamma_L))\, =\, {\mathcal S}(L)\, .
$$
Using Proposition \ref{prop1}(3) we know that $\text{Gal}
(\gamma_L)$ acts freely on $\text{Prym}(\gamma_L)$,
and
\begin{equation}\label{e17}
{\mathcal S}(L)\, =\, \text{Prym}(\gamma_L)/\text{Gal}
(\gamma_L)\, .
\end{equation}

We will now explicitly describe the action of $\Gamma$ on
${\mathcal S}(L)$.

The group $\Gamma$ (see \eqref{e1}) has the following action on
the abelian variety $\text{Prym}(\gamma_L)$ defined in
\eqref{e15}. Take any line bundle $\zeta\, \in\, \Gamma$.
For any $\eta\, \in\, \text{Prym}(\gamma_L)$, we have
$$
\bigwedge\nolimits^2 \gamma_{L*}(\eta\bigotimes \gamma^*_L \zeta)
\,=\, (\bigwedge\nolimits^2\gamma_{L*} \eta)
\bigotimes \zeta^{\otimes 2}
\, =\, \bigwedge\nolimits^2\gamma_{L*} \eta\, =\, \xi\, .
$$
Therefore, we have a morphism
\begin{equation}\label{e3z}
\phi'(\zeta)\, :\, \text{Prym}(\gamma_L)\, \longrightarrow\,
\text{Prym}(\gamma_L)
\end{equation}
defined by $\eta\, \longmapsto\, \eta\bigotimes
\gamma^*_L \zeta$. Let
\begin{equation}\label{e3a}
\phi'\, :\, \Gamma\, \longrightarrow\,
\text{Aut}(\text{Prym}(\gamma_L))
\end{equation}
be the homomorphism defined by $\zeta\, \longrightarrow
\, \phi'(\zeta)$. In other words, $\phi'$ defines
an action of $\Gamma$ on $\text{Prym}(\gamma_L)$.

The map $p_0$ in \eqref{e16} clearly commutes with the
actions of $\Gamma$ on ${\mathcal M}_\xi$ and
$\text{Prym}(\gamma_L)$ defined by $\phi$ (see \eqref{e3})
and $\phi'$ (see \eqref{e3a}) respectively. Also, the actions
of $\Gamma$ and $\text{Gal}(\gamma_L)$ (see \eqref{e17})
on $\text{Prym}(\gamma_L)$ commute. Hence
\begin{equation}\label{e18}
\text{Gal}(\gamma_L)\backslash \text{Prym}(\gamma_L)/\Gamma
\, =\, {\mathcal S}(L)/\Gamma\, .
\end{equation}
Note that since the group $\text{Gal}(\gamma_L)$ is abelian,
any right action of $\text{Gal}(\gamma_L)$ is also a
left action of $\text{Gal}(\gamma_L)$.

Consider the action of $\Gamma$ on $\text{Prym}(\gamma_L)$
constructed in \eqref{e3a}. In the proof of the first statement
in Proposition \ref{prop1} we noted that $\gamma^*_L L$
has a canonical trivialization. Therefore,
$$
\phi'(L)\, =\, \text{Id}_{\text{Prym}(\gamma_L)}\, .
$$
Take any $\zeta\, \in\, \Gamma \setminus\{L\, ,
{\mathcal O}_X\}$. Then
$\gamma^*_L\zeta$ is a nontrivial holomorphic line bundle
on $Y_L$. Consequently, the translation $\phi'({\zeta})$
in \eqref{e3z} is fixed point free. Using this it follows
that the quotient $\text{Prym}(\gamma_L)/\Gamma$ is an
abelian variety. In particular, the homomorphism
\begin{equation}\label{qish}
H^i(\text{Prym}(\gamma_L)/\Gamma,\, {\mathbb Q})\,
\longrightarrow\, H^i(\text{Prym}(\gamma_L),\, {\mathbb Q})
\end{equation}
induced by the quotient map
\begin{equation}\label{qisi}
\text{Prym}(\gamma_L)\, \longrightarrow\,
\text{Prym}(\gamma_L)/\Gamma
\end{equation}
is an isomorphism for all $i$.

The quotient map in \eqref{qisi} clearly intertwines the
actions of $\text{Gal}(\gamma_L)$ on $\text{Prym}(\gamma_L)$
and $\text{Prym}(\gamma_L)/\Gamma$. Hence the isomorphism
in \eqref{qish} also intertwines the actions of $\text{Gal}
(\gamma_L)$. In view of this, from \eqref{e18} we conclude that
$$
H^i({\mathcal S}(L)/\Gamma,\, {\mathbb Q})\, =\,
H^i(\text{Gal}(\gamma_L)\backslash \text{Prym}(\gamma_L),
\, {\mathbb Q})
$$
for all $i$.

Consider the action of $\text{Gal}(\gamma_L)$ on
$\text{Prym}(\gamma_L)$. We have an natural isomorphism
$$
H^i(\text{Gal}(\gamma_L)\backslash \text{Prym}(\gamma_L),
\, {\mathbb Q})\, =\, H^i(\text{Prym}(\gamma_L),
\, {\mathbb Q})^\sigma\, ,
$$
where $\sigma\, \in\, \text{Gal}(\gamma_L)$ is the
nontrivial element (see \eqref{sigma}), and
$$
H^i(\text{Prym}(\gamma_L),
\, {\mathbb Q})^\sigma\, \subset\, H^i(\text{Prym}(\gamma_L),
\, {\mathbb Q})
$$
is the subspace fixed pointwise by $\sigma$.
It can be shown that $\sigma$ acts on $H^1(\text{Prym}(\gamma_L),
\, {\mathbb Q})$ as multiplication by $-1$. Note that for the
action of $\sigma$ on $H^1(\text{Pic}^1(Y_L),\, \mathbb Q)$,
the invariant subspace
$H^1(\text{Pic}^1(Y_L),\, {\mathbb Q})^\sigma$
is identified with $H^1(\text{Pic}^1(X),\, \mathbb Q)$;
here $H^1(\text{Pic}^1(X),\, \mathbb Q)$ is considered
as a subspace of $H^1(\text{Pic}^1(Y_L),\, {\mathbb Q})$
using the homomorphism defined by $c\,\longmapsto\,
\widetilde{\gamma}^*_L c$, where
$$
\widetilde{\gamma}_L\, :\, \text{Pic}^1(X)\, \longrightarrow\,
\text{Pic}^1(Y_L)
$$
is the homomorphism defined by $\zeta\,\longmapsto\, 
{\gamma}^*_L \zeta$. The natural decomposition
$$
H^1(\text{Pic}^1(Y_L),\, {\mathbb Q})\, =\,
H^1(\text{Pic}^1(X),\, \mathbb Q)\, \bigoplus
H^1(\text{Prym}(\gamma_L),\, {\mathbb Q})\, ,
$$
is preserved by the action of $\sigma$, and it
acts on $H^1(\text{Pic}^1(X),\, \mathbb Q)$ and
$H^1(\text{Prym}(\gamma_L),\, {\mathbb Q})$
as multiplication by $1$ and $-1$ respectively.

Therefore, $\sigma$ acts on
$$
H^i(\text{Prym}(\gamma_L),\, {\mathbb Q})\, =\,
\bigwedge\nolimits^i H^1(\text{Prym}(\gamma_L),\, {\mathbb Q})
$$
as multiplication by $(-1)^i$.
Since $\text{Prym}(\gamma_L)$ is an abelian variety of
dimension $g-1$, we have $\dim H^1(\text{Prym}(\gamma_L),
\, {\mathbb Q})\, =\, 2(g-1)$. This completes the proof
of the proposition.
\end{proof}

For any $i\, \geq\, 0$, denote the $\mathbb Q$--vector space
$H^{i + 2\iota(L)} (\mathcal{S}(L)/\Gamma, \mathbb Q)$ by
$A^i(L)$, where $\iota(L)$ is the degree shift. We recall that
$\iota(L)\, =\, g-1$ if $L$ is nontrivial, and
$\iota({\mathcal O}_X)\, =\, 0$ (see Corollary \ref{cor1}).
For any nontrivial $L \,\in\, \Gamma$, from Proposition
\ref{prop2} we know that $A^*(L)$ is a graded vector space
over $\mathbb Q$ with $d_g(i)$ generators of degree
$i +2(g-1)$. Note that $A^*({\mathcal O}_X) \,=\,
H^* ({\mathcal M}_\xi/ \Gamma, \mathbb Q)$.
We get the following description of the Chen--Ruan
cohomology group (compare with \eqref{eCR}):
\begin{equation}\label{eA}
H_{CR}^* ({\mathcal M}_\xi/ \Gamma ,\, \mathbb Q)
\,=\, \bigoplus_{L \in \Gamma} A^*(L).
 \end{equation}

\subsection{Cohomology of ${\mathcal M}_\xi/\Gamma$}

Consider the action of $\Gamma$ on ${\mathcal M}_\xi$ given by
the homomorphism $\phi$ in \eqref{e3}. It is known that the
corresponding action on $H^*({\mathcal M}_\xi,\, {\mathbb Q})$
of $\Gamma$ is the trivial action \cite[p. 215, Theorem 1]{HN},
\cite[p. 578, Proposition 9.7]{AB}. Therefore,
the homomorphism
\begin{equation}\label{is0psi}
\psi^*\, :\, H^*({\mathcal M}_\xi/\Gamma,\, {\mathbb Q})\,
\longrightarrow\, H^*({\mathcal M}_\xi,\,
{\mathbb Q})
\end{equation}
induced by the quotient map
\begin{equation}\label{is1psi}
\psi\, :\, {\mathcal M}_\xi\, \longrightarrow\,
{\mathcal M}_\xi/\Gamma
\end{equation}
is an isomorphism.

There is a holomorphic universal vector bundle
$\widetilde{\mathcal E}
\, \longrightarrow\, X\times {\mathcal M}_\xi$. It is
universal in the sense that for each point $m\, \in\,
{\mathcal M}_\xi$, the holomorphic vector bundle over $X$
obtained by restricting $\widetilde{\mathcal E}$
to $X\times \{m\}$
is in the isomorphism defined by the point $m$ of the moduli
space. Any two universal vector bundles
over $X\times {\mathcal M}_\xi$ differ by tensoring
with a line bundle pulled back from ${\mathcal M}_\xi$.
Therefore, the vector bundle
\begin{equation}\label{cU}
{\mathcal U}\, :=\, \text{ad}(\widetilde{\mathcal E})\,
\subset\,\widetilde{\mathcal E}\bigotimes
\widetilde{\mathcal E}^*
\end{equation}
defined by the sheaf of trace zero endomorphisms is unique
up to an isomorphism.

Consider
\begin{equation}\label{c2U}
c_2({\mathcal U})\, \in\, H^4(X\times {\mathcal M}_\xi,\,
{\mathbb Q})\, ,
\end{equation}
where $\mathcal U$ is defined in \eqref{cU}. Using
K\"unneth decomposition,
$$
c_2({\mathcal U})\, \in\, \bigoplus_{i=0}^2 H^i(X,\,
{\mathbb Q}) \bigotimes H^{4-i}({\mathcal M}_\xi,\,
{\mathbb Q})\,=\, \bigoplus_{i=0}^2 H_{i}(X,\,
{\mathbb Q})^*
\bigotimes H^{4-i}({\mathcal M}_\xi,\, {\mathbb Q})\, .
$$
Therefore, $c_2({\mathcal U})$ gives a
$\mathbb Q$--linear homomorphism
\begin{equation}\label{H}
H\, :\, \bigoplus_{i=0}^2H_{i}(X,\, {\mathbb Q})\,
\longrightarrow\, \bigoplus_{i=0}^2
H^{4-i}({\mathcal M}_\xi,\, {\mathbb Q})
\end{equation}
such that $H(H_{i}(X,\, {\mathbb Q}))\, \subset
\, H^{4-i}({\mathcal M}_\xi,\, {\mathbb Q})$.

It is known that the image of the homomorphism $H$ in
\eqref{H} generates the entire cohomology algebra
$\bigoplus_{i>0}H^i({\mathcal M}_\xi, \, {\mathbb Q})$
\cite[p. 338, Theorem 1]{Ne} (see also
\cite[p. 581, Theorem 9.11]{AB}).

We noted earlier that the homomorphism $\psi^*$
in \eqref{is0psi} is an isomorphism. Let
\begin{equation}\label{H2}
\widetilde{H}\, :=\, (\psi^*)^{-1}\circ
H\, :\, \bigoplus_{i=0}^2H_{i}(X,\, {\mathbb Q})\,
\longrightarrow\, \bigoplus_{i=0}^2
H^{4-i}({\mathcal M}_\xi/\Gamma,\, {\mathbb Q})
\end{equation}
be the composition homomorphism. Therefore, the
image of $\widetilde{H}$ generates the cohomology
algebra of ${\mathcal M}_\xi/\Gamma$.

\section{The Chen--Ruan cohomology ring}

Take any nontrivial line bundle $L\, \in\, \Gamma
\setminus\{{\mathcal O}_X\}$ (see \eqref{e1}). Let
\begin{equation}\label{if}
f\, :\, {\mathcal S}(L)/\Gamma\, \longrightarrow\,
{\mathcal M}_\xi/\Gamma
\end{equation}
be the inclusion map. Let
\begin{equation}\label{if2}
f^*\, :\, H^*({\mathcal M}_\xi/\Gamma,\, {\mathbb Q})
\,\longrightarrow\, H^*({\mathcal S}(L)/\Gamma,\,
{\mathbb Q})
\end{equation}
be the pull back operation
by the map $f$ in \eqref{if}.

Let
\begin{equation}\label{iot}
p_1\, :\, \text{Prym}(\gamma_L)\, \longrightarrow\,
{\mathcal S}(L)
\end{equation}
be the quotient map (see \eqref{e17}). Note that the
homomorphism of cohomologies with coefficients in
$\mathbb Q$ induced by $p_1$ is
injective. In fact the pullback operation by $p_1$
identifies $H^i({\mathcal S}(L),\, {\mathbb Q})$
with the invariant part
$H^i(\text{Prym}(\gamma_L),\, {\mathbb
Q})^{\text{Gal}(\gamma_L)}$ for all $i$. Let
\begin{equation}\label{iot2}
\iota_0\, :\, \text{Prym}(\gamma_L)\, \hookrightarrow\,
\text{Pic}^1(Y_L)
\end{equation}
be the inclusion map (see \eqref{e15}). There is a
canonical polarization
$$
\Theta\, \in\, H^2(\text{Pic}^1(Y_L),\, {\mathbb Q})
$$
constructed using the cup product on
$H^1(Y_L,\, {\mathbb Q})$ and the orientation of $Y_L$.

\begin{proposition}\label{propi}
Consider the composition $f^*\circ \widetilde{H}$,
where $\widetilde H$ and $f$ are constructed in
\eqref{H2} and \eqref{if2} respectively. Then
$$
(f^*\circ \widetilde{H})(H_1(X,\, {\mathbb Q}))
\, =\, 0\, =\,
(f^*\circ \widetilde{H})(H_0(X,\, {\mathbb Q}))\, .
$$
Furthermore,
$$
p^*_1 ((f^*\circ \widetilde{H})([X]))\,=\,
2\iota^*_0\Theta\, ,
$$
where $[X]\, \in\, H_2(X,\, {\mathbb Z})$ is the
oriented generator and $\Theta\, \in\, H^2(
{\rm Pic}^1(Y_L),\, {\mathbb Q})$ is the canonical
polarization; the maps $p_1$ and $\iota_0$ are
constructed in \eqref{iot} and \eqref{iot2}
respectively.
\end{proposition}

\begin{proof}
Consider the map $p_0$ constructed in \eqref{e16}.
For any $i\, \geq\, 0$, let
\begin{equation}\label{of2}
\widetilde{p}_i\, :\, H^i({\mathcal M}_\xi\,
{\mathbb Q})\,\longrightarrow\,
H^i(\text{Prym}(\gamma_L),\, {\mathbb Q})
\end{equation}
be the homomorphism defined by
$c\, \longmapsto\, p^*_0 c$. We noted that
the homomorphism $\psi^*$
in \eqref{is0psi} is an isomorphism.
We also observed that the homomorphism in \eqref{qish}
is an isomorphism. Therefore, to
prove the proposition it is enough to show that
the following three are valid:
\begin{equation}\label{ch1}
\widetilde{p}_3(H(H_1(X,{\mathbb Q})))\, =\,0\, ,
\end{equation}
\begin{equation}\label{ch2}
\widetilde{p}_4(H(H_0(X,{\mathbb Q})))\, =\,0\, ,
\end{equation}
and
\begin{equation}\label{ch3}
\widetilde{p}_2(H([X]))\, =\,2\iota^*_0\Theta\, ,
\end{equation}
where $H$ is the homomorphism in \eqref{H},
and $\widetilde{p}_i$ is constructed in \eqref{of2}
(the map $\iota_0$ is defined in \eqref{iot2}).

Fix a universal (Poincar\'e) line bundle
\begin{equation}\label{plb}
{\mathcal L}_0\, \longrightarrow\, Y_L\times
\text{Pic}^1(Y_L)\, ,
\end{equation}
where $Y_L$ is the covering in \eqref{e4}. Let
\begin{equation}\label{Pl}
{\mathcal L}\, :=\, (\text{Id}_{Y_L}\times
\iota_0)^*{\mathcal L}_0\, \longrightarrow\,
Y_L\times\text{Prym}(\gamma_L)
\end{equation}
be the line bundle, where $\iota_0$ in the
inclusion map in \eqref{iot2}.

Consider the vector bundle
$$
(\gamma_L\times p_1)^*{\mathcal U}\,
\longrightarrow\, Y_L\times \text{Prym}
(\gamma_L)\, ,
$$
where $\mathcal U$ is the vector bundle in \eqref{cU},
and $p_1$ is the map in \eqref{iot} (recall
that ${\mathcal S}(L)\, \subset\, {\mathcal M}_\xi$).
It is straight forward to check that
\begin{equation}\label{chiis}
(\gamma_L\times p_1)^*{\mathcal U}\,=\,
({\mathcal L}^*\bigotimes (\sigma\times
\text{Id}_{\text{Prym}(\gamma_L)})^*
{\mathcal L}) \bigoplus ({\mathcal L}\bigotimes
(\sigma\times\text{Id}_{\text{Prym}(\gamma_L)})^*
{\mathcal L}^*)\bigoplus {\mathcal O}_{Y_L\times
\text{Prym}(\gamma_L)}\, ,
\end{equation}
where $\mathcal L$ is the line bundle in
\eqref{Pl}, and $\sigma$ is the automorphism in
\eqref{sigma}. From \eqref{chiis},
\begin{equation}\label{c2s}
(\gamma_L\times p_1)^* c_2({\mathcal U})\,=\, -
((\sigma\times\text{Id}_{\text{Prym}(\gamma_L)})^*
c_1({\mathcal L}) - c_1({\mathcal L}))^2
\end{equation}

Consider the Abel--Jacobi map $Y_L\, \longrightarrow\,
\text{Pic}^1(Y_L)$ defined by $y\, \longmapsto\,
{\mathcal O}_{Y_L}(y)$. The corresponding homomorphism
\begin{equation}\label{tB}
\widetilde{B}\, :\,
H^1(\text{Pic}^1(Y_L),\, {\mathbb Q})
\,\longrightarrow\, H^1(Y_L,\, {\mathbb Q})\, =\,
H^1(Y_L,\, {\mathbb Q})^*
\end{equation}
is an isomorphism; the identification of
$H^1(Y_L,\, {\mathbb Q})$ with $H^1(Y_L,\, {\mathbb Q})^*$
is given by the cup product on $H^1(Y_L,\, {\mathbb Q})$.
Let
\begin{equation}\label{tB1}
B\, \in\, H^1(Y_L,\, {\mathbb Q})\bigotimes
H^1(\text{Pic}^1(Y_L),\, {\mathbb Q})\,\subset\,
H^2(Y_L\times\text{Pic}^1(Y_L),\, {\mathbb Q})
\end{equation}
be the element given by the isomorphism $\widetilde{B}$
in \eqref{tB}.

The Poincar\'e line bundle ${\mathcal L}_0$ in
\eqref{plb} can be so normalized that
$$
c_1({\mathcal L}_0)\, =\, p^*_{Y_L}([Y_L])+ B\, ,
$$
where
\begin{itemize}
\item $p_{Y_L}\, :\, Y_L\times \text{Pic}^1(Y_L)
\, \longrightarrow\, Y_L$ is the projection,
and $[Y_L]\, \in\, H^2(Y_L,\, {\mathbb Z})$
is the oriented generator, and
\item $B$ is the cohomology class in \eqref{tB1}
\end{itemize}
(see \cite[Ch. 1, \S~5]{ACGH} and \cite[Ch. IV, \S~2]{ACGH}).

Using the above description of
$c_1({\mathcal L}_0)$ together with \eqref{c2s}
we conclude that \eqref{ch2} holds (recall
that $\psi^*$ in \eqref{is0psi} is an isomorphism).

The involution $\sigma$ in \eqref{sigma} defines
actions of ${\mathbb Z}/2\mathbb Z$ on
$H^1(Y_L,\, {\mathbb Q})$ and
$\text{Pic}^1(Y_L)$. The action of
${\mathbb Z}/2\mathbb Z$ on $\text{Pic}^1(Y_L)$
induces an action
of ${\mathbb Z}/2\mathbb Z$ on
$H^1(\text{Pic}^1(Y_L),\, {\mathbb Q})$. The
homomorphism $\widetilde B$ in\eqref{tB} intertwines
the actions of ${\mathbb Z}/2\mathbb Z$ on
$H_1(Y_L,\, {\mathbb Q})$ and
$H^1(\text{Pic}^1(Y_L),\, {\mathbb Q})$. We also
note that
\begin{equation}\label{subspacei}
H^1(Y_L,\, {\mathbb Q})^\sigma \, =\,
H^1(X,\, {\mathbb Q})\, ,
\end{equation}
and the subspace $H^1(\text{Prym}(\gamma_L),\,
{\mathbb Q})\, \subset\,
H^1(\text{Pic}^1(Y_L),\, {\mathbb Q})$ coincides
with the subspace on which the nonzero element
in ${\mathbb Z}/2\mathbb Z$ acts as multiplication
by $-1$ (this was also noted in the proof of
Proposition \ref{prop2}).

Since $\widetilde B$ in\eqref{tB} intertwines
the actions of ${\mathbb Z}/2\mathbb Z$, it sends
the invariant
subspace $H^1(\text{Pic}^1(Y_L),\, {\mathbb
Q})^{{\mathbb Z}/2\mathbb Z}$ to
the subspace in \eqref{subspacei}. Using this and
\eqref{c2s} we now conclude that \eqref{ch1} holds.

To prove \eqref{ch3}, we will first recall a
description of the cohomology class
$$
H([X])\, \in\, H^2({\mathcal M}_\xi,\,
{\mathbb Q})\, ,
$$
where $H$ is constructed in \eqref{H}.

Let $\widetilde{\mathcal E}$ be a universal vector bundle
over $X\times {\mathcal M}_\xi$ (see \eqref{cU}).
Let
\begin{equation}\label{pM}
p_M\, :\, X\times {\mathcal M}_\xi\,\longrightarrow
\, {\mathcal M}_\xi
\end{equation}
be the projection. Define the line bundle
$$
\text{Det}(\widetilde{\mathcal E})\,:=\,
(\bigwedge\nolimits^{\text{top}}R^0p_{M*}\widetilde{\mathcal E}
)^*\bigotimes (\bigwedge\nolimits^{\text{top}}R^1p_{M*}
\widetilde{\mathcal E})
\, \longrightarrow\, {\mathcal M}_\xi\, .
$$
Fix a point $x_0\, \in\, X$. Let
$$
\widetilde{\mathcal E}_{x_0}\, :=\,
\widetilde{\mathcal E}\vert_{\{x_0\}
\times {\mathcal M}_\xi}\,\longrightarrow\,
{\mathcal M}_\xi
$$
be the vector bundle over ${\mathcal M}_\xi$.
Now define the line bundle
$$
\Theta_M\, :\,
\text{Det}(\widetilde{\mathcal E})^{\otimes 2}
\bigotimes (\bigwedge\nolimits^2\widetilde{\mathcal E}_{x_0}
)^{\otimes (3-2g)}\, \longrightarrow\, {\mathcal M}_\xi\, .
$$
Both $\text{Det}(\widetilde{\mathcal E})$ and
$\bigwedge^2\widetilde{\mathcal E}_{x_0}$ depend
on the choice
of $\widetilde{\mathcal E}$, but $\Theta_M$ is
independent of the choices of $\widetilde{\mathcal E}$
and $x_0$. In fact, the line bundle $\Theta_M$ is the
ample generator of $\text{Pic}({\mathcal M}_\xi)\,
\cong\, \mathbb Z$.

Since $T{\mathcal M}_\xi\, =\, R^1p_{M*}{\mathcal U}$,
where $p_M$ is the projection in \eqref{pM}, from
the Hirzebruch--Riemann--Roch theorem it follows that
$$
H([X])\, =\, c_1(T{\mathcal M}_\xi)
$$
(note that $R^1p_{{\mathcal M}*}
{\mathcal U}\, =\, 0$). Hence we have
\begin{equation}\label{ichf}
H([X])\, =\, 2\cdot c_1(\Theta_M)\, ,
\end{equation}
where $H$ is constructed in \eqref{H} (see
\cite[p. 69, Theorem 1]{Ra} and \cite[p. 338,
(1)]{Ne}).

We will now recall a similar description of the
cohomology class $\Theta$ on $\text{Pic}^1(Y_L)$.

Take a Poincar\'e line bundle ${\mathcal L}_0$
on $Y_L\times\text{Pic}^1(Y_L)$ (see \eqref{plb}).
Let $P_J$ denote the projection of $Y_L\times
\text{Pic}^1(Y_L)$ to $\text{Pic}^1(Y_L)$. Let
\begin{equation}\label{Lx0}
{\mathcal L}_{x_0}\, :=\, {\mathcal L}_0\vert_{\{x_0\}
\times\text{Pic}^1(Y_L)}\, \longrightarrow\,
\text{Pic}^1(Y_L)
\end{equation}
be the line bundle, where $x_0$ as before is a fixed
point of $X$. Now define the line bundle
$$
\Theta_J\, :=\,
\text{Det}({\mathcal L}_0)\bigotimes
{\mathcal L}^{2-g}_{x_0}\, =\,
(\bigwedge\nolimits^{\text{top}}R^0p_{J*}{\mathcal L}_0)^*
\bigotimes
(\bigwedge\nolimits^{\text{top}}R^1p_{J*}{\mathcal L}_0)
\bigotimes {\mathcal L}^{2-g}_{x_0}
\, \longrightarrow\, \text{Pic}^1(Y_L)\, .
$$
This line bundle $\Theta_J$
does not depend on the choice
of ${\mathcal L}_0$, but it depends on the choice
of $x_0$. Since $X$ is connected,
$$
c_1(\Theta_J)\, \in\, H^2(\text{Pic}^1(Y_L),\,
{\mathbb Q})
$$
is independent of $x_0$. It is known that
\begin{equation}\label{ichf2}
c_1(\Theta_J)\, =\, \Theta\, .
\end{equation}

Let ${\mathcal F}_0\, :=\, (\gamma_L\times 
\text{Id}_{\text{Pic}^1(Y_L)})_*{\mathcal L}_0\,\longrightarrow
\, X\times\text{Pic}^1(Y_L)$ be the vector bundle.
Using \eqref{isi1} we have an isomorphism of line bundles
$$
\text{Det}({\mathcal L}_0)\, =\,\text{Det}({\mathcal F}_0)
\,:=\, (\bigwedge\nolimits^{\text{top}}R^0q_{X*}{\mathcal F}_0)^*
\bigotimes (\bigwedge\nolimits^{\text{top}}R^1q_{X*}{\mathcal F}_0)
\, \longrightarrow\, \text{Pic}^1(Y_L)\, ,
$$
where $q_X$ is the projection of $X\times\text{Pic}^1(Y_L)$
to $\text{Pic}^1(Y_L)$. We may choose ${\mathcal L}_0$ and
$x_0$ such that the line bundle ${\mathcal L}_{x_0}$ (see
\eqref{Lx0}) is trivial. Hence comparing \eqref{ichf} and 
\eqref{ichf2} we conclude that \eqref{ch3} holds. This
completes the proof of the proposition. \end{proof}

Recall that we denoted $H^{* + 2\iota(L)}
(\mathcal{S}(L)/\Gamma, \mathbb Q)$ by
$A^*(L)$, which is a graded vector space over
$\mathbb Q$ with
$d_g(i)$ generators of degree $i +2(g-1)$ for
every nontrivial $L
\in \Gamma$, and $ A^*({\mathcal O}_X) \,=\,
H^* ({\mathcal M}_\xi/ \Gamma, \mathbb Q)$.
There is a nondegenerate bilinear Poincar\'e
pairing $\langle\, , \rangle$ for
Chen--Ruan cohomology. For $\alpha \,\in\, A^*(L)$
and $\beta \in A^*(L')$, the pairing
$\langle \alpha\, , \beta\rangle$ is nonzero only
when $L' \,=\, L^{-1}\,=\, L$. In this case it is defined by
\begin{equation}\label{ePP}
\langle\alpha\, , \beta\rangle \,=\,
\int_{\mathcal{S}(L)/\Gamma}^{\rm orb} \alpha \bigwedge \beta
\end{equation}

Here, and henceforth, we use $\bigwedge$ to represent ordinary cup 
product. The integral notation 
$\int_{\mathcal{Y}/G}^{\rm orb}$
refers to a multiple of the evaluation on the fundamental class of
$\mathcal{Y}/G$. This multiple is the reciprocal of 
the cardinality of the subgroup of $G$ that acts trivially on the 
manifold $\mathcal{Y}$ (see page 6 of \cite{CR}).
We avoid differential forms unlike \cite{CR} since we
have the coefficients to be $\mathbb Q$.

For $\alpha_1 \,\in\, A^p(L_1)$, $\alpha_2 \,\in\,
A^q(L_2)$, the Chen--Ruan product
$$
\alpha_1 \bigcup \alpha_2
\,\in\, A^{p+q}(L_1 \bigotimes L_2)
$$
is defined via the relation
\begin{equation}\label{eTP}
\langle\alpha_1 \bigcup \alpha_2\, , \alpha_3\rangle\, =\,
\int_{\mathbf{S}/\Gamma}^{\rm orb} e_1^*\alpha_1 \bigwedge e_2^*
\alpha_2 \bigwedge
e_3^*\alpha_3 \bigwedge c_{\rm top} \mathcal{F}
\end{equation}
for all $\alpha_3 \,\in\, A^*(L_3)$ (it is
enough to consider $L_3\, =\, L_1\bigotimes L_2$), where
$$
\mathbf{S}\, :=\, \bigcap_{i=1}^3 \mathcal{S}(L_i)
$$
and $e_i\, :\, \mathbf{S}/\Gamma\, \longrightarrow\,
\mathcal{S}(L_i)/\Gamma$
are the canonical inclusions. Here $\mathcal F$ is a complex
$\Gamma$--bundle over $\mathbf S$
(or equivalently an orbifold vector bundle over
$\mathbf{S}/\Gamma$) of rank
\begin{equation}\label{erank}
{\rm rank} (\mathcal{F}) \,=\,
\dim_{\mathbb Q} \mathbf{S} -
\dim_{\mathbb Q} \mathcal{M}_\xi + \sum_{j=1}^3 \iota(L_j)
\end{equation}
(see the proof of Theorem 4.1.5 in \cite{CR}). In general,
$c_{\rm top} \mathcal{F}$ (as defined in \cite{CR}) 
is $\mathbb{R}$-valued, but we will see below that it is 
$\mathbb{Q}$-valued in our 
 case.
 
If $L_1\,=\, L_2\,=\, \mathcal{O}_X$, then the Chen--Ruan product
$\alpha_1\bigcup \alpha_2$
is the ordinary cup product in
$H^* ({\mathcal M}_\xi(r)/ \Gamma,\, \mathbb Q)$.

Since $L_3\, =\, L_1\bigotimes L_2$, we
only need to consider the following remaining cases:
\begin{itemize}
\item[a)] $ L_1= L_2= L \neq \mathcal{O}_X, L_3 = \mathcal{O}_X$
\item[b)] $L_1 = L \neq \mathcal{O}_X, L_2 = \mathcal{O}_X, L_3 = L$
\item[c)] $L_1= \mathcal{O}_X, L_2 = L \neq \mathcal{O}_X, L_3 = L$
\item[d)] $L_1 \neq \mathcal{O}_X, L_2
\neq\mathcal{O}_X, L_1 \neq L_2, L_3= L_1 \bigotimes L_2$.
\end{itemize}

Part of the calculations for the first three cases are analogous.
In these cases, $\mathbf{S}= \mathcal{S}(L)$, and by Corollary
\ref{cor1},
\begin{equation}
{\rm rank}({\mathcal F}) \,=\, (g-1) - 3(g-1) + 2(g-1)\,=\,0
\end{equation}
so that \eqref{eTP} reduces to
\begin{equation}
 \langle \alpha_1 \bigcup \alpha_2\, , \alpha_3 \rangle\, =\, 
\int_{\mathcal{S}(L)/\Gamma}^{\rm orb}
e_1^*\alpha_1 \bigwedge e_2^* \alpha_2 \bigwedge
e_3^*\alpha_3\, .
\end{equation}

\subsection{Case $a)$}

In this case \eqref{eTP} becomes
\begin{equation}
 \langle\alpha_1 \bigcup \alpha_2\, , \alpha_3 \rangle\, =\, 
\int_{\mathcal{S}(L)/\Gamma}^{\rm orb}
\alpha_1 \bigwedge\alpha_2 \bigwedge
e_3^*\alpha_3\, ,
\end{equation}
and $e_3$ coincides with the inclusion map $f$ in \eqref{if}.

Let us define $\kappa$ to be the cohomology class
\begin{equation}\label{eqkappa}
\kappa \,= \, \widetilde{H}[X]\,\in\, H^2({\mathcal M}_\xi/
\Gamma,\, \mathbb Q) \,=\, A^2(\mathcal{O}_X) 
\end{equation}
where $\widetilde{H}$ is constructed in \eqref{H2}. We have
\begin{equation}\label{el}
p^*_1q^*f^*(\kappa) \,= \, 2\iota^*_0\Theta
\end{equation}
(see Proposition \ref{propi}), where $f^*$, $p_1$ and
$\iota_0$ are constructed in \eqref{if2}, \eqref{iot} and
\eqref{iot2} respectively, and $q\, :\, {\mathcal 
S}(L)\,\longrightarrow\, {\mathcal S}(L)/\Gamma$ is the
quotient map.

{}From Proposition \ref{propi}
we have $e_3^{\ast} \alpha_3 = 0$ unless $\alpha_3$ is a linear 
combination of $\{\kappa^m\}_{m=0}^{g-1}$. Since
$\alpha_1 \bigwedge \alpha_2 \in H^{p+q - 4(g-1)} 
(\mathcal{S}(L)/\Gamma, \,{\mathbb Q})$, we know that
$\langle\alpha_1 \bigcup \alpha_2\, , \alpha_3\rangle$ is nonzero 
only if $\alpha_3$ is a multiple of
$\kappa^{m_0} $ where $m_0 = {3(g-1)-(p+q)/2}$. The class
$\alpha_1 \bigwedge \alpha_2 \bigwedge f^{\ast} \kappa^{m_0} $ is
some multiple $c(\alpha_1, \alpha_2) \,\Omega$ of the normalized
top degree cohomology class $\Omega$ of $\mathcal{S}(L)/\Gamma$
satisfying
$$
\int_{ \mathcal{S}(L)/\Gamma}^{\rm orb} \Omega \,= \,1\, .
$$
This constant $c(\alpha_1, \alpha_2)$ can be computed
because we know $f^{\ast}(\kappa)$ in terms of
the generators of $H^{\ast}(\mathcal{S}(L)/\Gamma,\,
{\mathbb Q})$ (see \eqref{el}). We obtain
\begin{equation}
\langle\alpha_1 \bigcup \alpha_2\, , \alpha_3\rangle \,=\, 
\left\{ \begin{array}{ll}
c(\alpha_1, \alpha_2) \,d & \mbox { if }\, \alpha_3 = d \, 
\kappa^{m_0} \\
\,0 & \,\mbox { otherwise. }
\end{array} \right. \end{equation}

In the present case, from \eqref{ePP},
\begin{equation}
\int_{ {\mathcal M}_\xi/ \Gamma} (\alpha_1 \bigcup 
\alpha_2) \bigwedge \alpha_3
\,=\,\left\{ \begin{array}{ll} c(\alpha_1, \alpha_2) \, d
 & \mbox { if } \alpha_3 = d \, \kappa^{m_0} \\
& \\
 0 & \mbox { otherwise. }
 \end{array} \right.
\end{equation}
Hence we obtain
\begin{equation}\label{eqcasea}
\alpha_1 \bigcup \alpha_2 \,=\,\frac{c (\alpha_1, \alpha_2)} 
{v} \kappa^{m_1}\, ,
\end{equation}
where $m_1 = 3(g-1) - m_0 = \frac{p+q}{2}$ and
\begin{equation}\label{eqcl}
v = \int_{ {\mathcal M}_\xi/ \Gamma} 
\kappa^{3(g-1)}\, .
\end{equation}

Consider $H$ constructed in \eqref{H}.
Thaddeus calculated that
$$
\int_{{\mathcal M}_\xi} H([X])^{3g-3}\, =\,\frac{(3g-3)!}{(2g-2)!}
2^{2g-2}(2^{2g-2} -2)|B_{2g-2}|\, ,
$$
where $B_{2g-2}$ is the Bernoulli number (see
\cite[p. 147, (29)]{Th} and the line following it).
Note that $v$ in \eqref{eqcl} satisfies the condition
$$
v\, =\, \frac{1}{2^{2g}}\int_{{\mathcal M}_\xi} H([X])^{3g-3}\, .
$$

\subsection{Case $b)$}

In this case \eqref{eTP} becomes
\begin{equation}\label{eqcaseb1}
\langle\alpha_1 \bigcup \alpha_2\, , \alpha_3 \rangle\, =\,
\int_{\mathcal{S}(L)/\Gamma}^{\rm orb}
\alpha_1 \bigwedge e_2^{\ast} \alpha_2 \bigwedge \alpha_3\, ,
\end{equation}
where $e_2$ is the inclusion $f: \mathcal{S}(L)/\Gamma \to {\mathcal 
M}_\xi/ \Gamma $. Comparing \eqref{eqcaseb1} with \eqref{ePP} we get
\begin{equation}\label{eqcaseb2}
\int_{\mathcal{S}(L)/\Gamma}^{\rm orb}
(\alpha_1 \bigcup \alpha_2) \bigwedge \alpha_3
\,= \,\int_{\mathcal{S}(L)/\Gamma}^{\rm orb}
( \alpha_1 \bigwedge f^{\ast} \alpha_2) \bigwedge \alpha_3
\end{equation}
for all $\alpha_3$. Thus we deduce
\begin{equation}\label{eqcaseb3}
\alpha_1 \bigcup \alpha_2\,=\,\alpha_1 \bigwedge f^{\ast} \alpha_2
\end{equation}
Note that $f^{\ast} \alpha_2 \,=\,0$ unless $\alpha_2$ is
scalar multiple of a power of $\kappa$.
 
\subsection{Case $c)$}

By an argument very similar to case $b)$, we get
\begin{equation}\label{eqc3}
\alpha_1 \bigcup \alpha_2 \,=\, f^{\ast}\alpha_1 \bigwedge\alpha_2
\end{equation}
 
\subsection{Case $d)$}

We invoke Proposition \ref{prop int}.
If $\mu(\omega(L_1) \bigotimes \omega(L_2)) = 0$ then
$$
\mathbf{S}\,=\, \mathcal{S}(L_1)\bigcap \mathcal{S}(L_2)
\,=\, \emptyset
$$
and consequently the Chen--Ruan product
$$
\alpha_1 \bigcup \alpha_2 \,=\, 0
$$
for all $\alpha_i \in A^{\ast}(L_i)$, $i\,=\, 1,2$.
On the other hand, if $\mu(\omega(L_1) \bigotimes \omega(L_2))\,
=\, 1$, then $\mathbf{S}/\Gamma$ is a point
modulo a finite group of order $4$. For
dimensional reasons, we have $c_{\rm top} \mathcal F = 1$ and
\begin{equation}
\langle\alpha_1 \bigcup \alpha_2\, , \alpha_3 \rangle \,=\,
\left\{ \begin{array}{ll} \frac{1}{4} \alpha_1 \alpha_2 \alpha_3
 & \mbox { if } \alpha_i \in A^{2g-2}(L_i) \,~\, \forall\, i \\
 &\\
 0 & \mbox { otherwise. }
 \end{array} \right. \end{equation}

Therefore by \eqref{ePP}, if $\Omega^{\prime}$ denotes
the normalized top degree cohomology class on
$\mathcal{S}(L_1 \bigotimes L_2)/\Gamma$ such that
$$
\int_{\mathcal{S}(L_1 \bigotimes L_2)/\Gamma}^{\rm orb}
\Omega^{\prime}\,=\, 1\, ,
$$
then we have
\begin{equation}
\alpha_1 \bigcup \alpha_2 \,=\,\left\{
\begin{array}{ll} \frac{1}{4} \alpha_1 \alpha_2 \Omega^{\prime}
& \mbox { if } \alpha_i \in A^{2g-2}(L_i) \,~\, \forall\, i \\
&\\
0 & \mbox { otherwise. }
\end{array} \right. \end{equation}

\medskip
\noindent
\textbf{Acknowledgements.}\, This work was carried out during
a visit of the first author to ISI and a visit of the second
author to TIFR. We record our thanks to these two institutes.

%%%%%%%%%%%%%%%%%%%%%%%%%%%%%%%%%%%%%%%%%%%%%%%%%%%%%%%%%%%%%%%%%


\begin{thebibliography}{AAAA}

\bibitem[AR]{AR} A. Adem and Y. Ruan, Twisted orbifold $K$--theory,
Comm. Math. Phys. \textbf{237} (2003), 533--556.

\bibitem[ACGH]{ACGH} E. Arbarello, M. Cornalba, P. A. Griffiths and
J. Harris, \textit{Geometry of algebraic curves}, Volume I,
Grundlehren der Mathematischen Wissenschaften, 267,
Springer-Verlag, New York, 1985.

\bibitem[At]{At} M. F. Atiyah, On the Krull--Schmidt theorem with
application to sheaves, Bull. Soc. Math. Fr. \textbf{84}
(1956), 307--317.

\bibitem[AB]{AB} M. F. Atiyah and R. Bott, The Yang--Mills
equations over Riemann surfaces, Phil. Trans. Roy. Soc. Lond.
\textbf{308} (1982), 523--615.

\bibitem[BNR]{BNR} A. Beauville, M. S. Narasimhan and S. Ramanan,
Spectral curves and the generalised theta divisor, Jour.
Reine Angew. Math. \textbf{398} (1989), 169--179.

\bibitem[CR1]{CR} W. Chen and Y. Ruan, A new cohomology theory of
orbifold, Comm. Math. Phys. \textbf{248} (2004), 1--31.

\bibitem[CR2]{CR2} W. Chen and Y. Ruan, Orbifold Gromov-Witten
theory, in: \textit{Orbifolds in mathematics and physics}
(Madison, WI, 2001), pp. 25--85, Contemp.
Math. \textbf{310}, Amer. Math. Soc., Providence, RI, 2002.

\bibitem[FG]{FG} B. Fantechi and L. G\"ottsche, Orbifold
cohomology for global quotients, Duke Math. Jour.
\textbf{117} (2003), 197--227.

\bibitem[GK]{GK} V. Ginzburg and D. Kaledin, Poisson
deformations of symplectic quotient singularities,
Adv. Math. \textbf{186} (2004), 1--57.

\bibitem[Hi]{Hi} N. J. Hitchin, Stable bundles and integrable
systems, Duke Math. Jour. \textbf{54} (1987), 91--114.

\bibitem[HN]{HN} G. Harder and M. S. Narasimhan, On the
cohomology groups of moduli spaces of vector bundles on
curves, Math. Ann. \textbf{212} (1975), 215--248.

\bibitem[LP]{LP} E. Lupercio and M. Poddar, The global
McKay--Ruan correspondence via motivic integration,
Bull. London Math. Soc. \textbf{36} (2004), 509--515.

\bibitem[Mu1]{Mu1} D. Mumford, \textit{Abelian Varieties},
Tata Institute of Fundamental Research Studies in Mathematics,
No. 5, Oxford University Press, London 1970.

\bibitem[Mu2]{Mu2} D. Mumford, Prym varieties. I, in:
\textit{Contributions to analysis (a collection of papers
dedicated to Lipman Bers)}, pp. 325--350.
Academic Press, New York, 1974.

\bibitem[Ne]{Ne} P. E. Newstead, Characteristic classes of
stable bundles of rank $2$ over an algebraic curve, Trans.
Amer. Math. Soc. \textbf{169} (1972), 337--345.

\bibitem[Ra]{Ra} S. Ramanan, The moduli spaces of vector
bundles over an algebraic curve, Math. Ann. \textbf{200}
(1973), 69--84.

\bibitem[Ru]{Ru} Y. Ruan, Stringy geometry and topology of
orbifolds, in: \textit{Symposium in Honor of C. H. Clemens}
(Salt Lake City, UT, 2000), pp. 187--233, Contemp.
Math. \textbf{312}, Amer. Math. Soc., Providence, RI, 2002. 

\bibitem[Th]{Th} M. Thaddeus, Conformal field theory and
the cohomology of the moduli space of stable bundle,
Jour. Diff. Geom. \textbf{35} (1992), 131--149.

\bibitem[Ur]{Ur} B. Uribe, Orbifold cohomology of the
symmetric product, Comm. Anal. Geom. \textbf{13}
(2005), 113--128.

\bibitem[Ya]{Ya} T. Yasuda, Twisted jets, motivic measures and
orbifold cohomology, Compos. Math. \textbf{140} (2004), 396--422.

\end{thebibliography}
\end{document}